\documentclass[a4paper,11pt,oneside,reqno]{amsart}
\usepackage{amsthm,amsmath,amssymb}
\usepackage{hyperref}
\usepackage{orcidlink}
\usepackage{graphicx}
\usepackage{caption}
\newcommand{\doi}[1]{\href{https://doi.org/#1}{doi:#1}}
\usepackage{tikz}

\theoremstyle{definition}
\newtheorem{definition}{Definition}
\newtheorem{example}[definition]{Example}
\theoremstyle{remark}
\newtheorem{remark}[definition]{Remark}
\newtheorem{openproblem}[definition]{Open problem}
\theoremstyle{plain}
\newtheorem{theorem}[definition]{Theorem}
\newtheorem{lemma}[definition]{Lemma}

\newcommand{\Z}{\mathbb{Z}}
\newcommand{\N}{\mathbb{N}}
\newcommand{\R}{\mathbb{R}}
\newcommand{\RP}{\mathbb{RP}^1}

\begin{document}

\title[Lattice-line directions]{On the directions occurring in lattice-line coverings of the integer plane}
\author{Jan Snellman \orcidlink{0009-0002-6676-5068}}
\address{Matematiska Institutionen, Linköpings Universitet, 581 83 Linköping, Sweden}
\email{jan.snellman@liu.se}
\makeatletter
\let\addresses\@empty
\makeatother
\maketitle

\begin{center}
  \small
  Matematiska Institutionen, Linköpings Universitet, 581\,83 Linköping, Sweden\\
  \texttt{jan.snellman@liu.se}
\end{center}

\begin{abstract}
  We consider families of lines that cover every point of the integer lattice \(\Z^2\), subject to
  the constraint that no two lines of different direction in the family meet at a lattice point.
  Restricting to \emph{lattice lines} (lines containing at least two, hence infinitely many,
  lattice points, equivalently of rational direction), we show that the set of directions
  occurring in such a covering can be made dense in the space of line directions. The construction
  is a recursive splitting of \(\Z^2\) into nested rank-2 sublattice cosets, each handed off to a
  freshly chosen direction; the key technical point is a steering lemma showing that at every
  stage of the recursion a new direction arbitrarily close to any prescribed target can still be
  realized, via an elementary sieve bound.
\end{abstract}

\begin{center}
\footnotesize
\emph{Note.} An earlier version of this work appeared as the technical report
\cite{report2026}. The present manuscript extends it; see \texttt{changes-wrt-report.md} in the
source repository \cite{gitrepo} for a running account of what has been added or altered since.
\end{center}

\section{Introduction}

Consider the integer lattice \(\Z^2\subset\R^2\). We wish to cover every point of \(\Z^2\) by a family
\(\mathcal F\) of lines, subject to: any two lines of \(\mathcal F\) may cross, but never at a point of
\(\Z^2\). Which sets of line-directions can occur across such a family?

\begin{remark}[The lattice-line hypothesis]
  \label{rem:hypothesis}
  As posed, with no further restriction on the lines allowed, the question is trivial: for each
  \(z\in\Z^2\) take the line \(L_z\) through \(z\) of irrational slope \(\alpha_z\). Then
  \begin{equation}
    \label{eq:irrational-line}
    L_z\cap\Z^2=\{z\}
  \end{equation}
  (a second lattice point on \(L_z\) would force \(\alpha_z\in\mathbb Q\)), so distinct such lines
  never share a lattice point, and choosing the \(\alpha_z\) to range over a dense (indeed, all) set of
  irrational directions already realizes an uncountable dense set of directions. Everything of
  interest in this manuscript therefore concerns the restriction to \emph{lattice lines}: lines
  containing at least two (hence, by the same argument, infinitely many) points of \(\Z^2\),
  equivalently lines of rational direction. This restriction is in force throughout.
\end{remark}

We now formalize the objects of study as explicit definitions.

\begin{definition}[Primitive direction]
  \label{def:direction}
  A \emph{primitive direction} is a pair \(d=(p,q)\in\Z^2\) with \(\gcd(p,q)=1\). Two primitive
  directions \(d,d'\) are \emph{equal as directions} if \(d'=\pm d\). Write
  \begin{equation}
    \label{eq:phi-def}
    \varphi_d(x,y) = qx-py,
  \end{equation}
  a surjective homomorphism \(\Z^2\to\Z\) depending only on \(d\) up to sign of the whole map (in
  particular \(\ker\varphi_d=\ker\varphi_{-d}\)). The space of directions is \(\RP=S^1/\{\pm1\}\),
  parametrized by \(\theta\in[0,\pi)\); we metrize it by
  \begin{equation}
    \label{eq:RP-metric}
    d(\theta,\theta') = \min(|\theta-\theta'|,\,\pi-|\theta-\theta'|),
  \end{equation}
  so ``within \(\varepsilon\) of \(\theta\)'' always means within \(\varepsilon\) in this metric. Since
  \([0,\pi)\to\RP\) is a continuous surjection, a set dense in \([0,\pi)\) is automatically dense in
  \(\RP\).
\end{definition}

\begin{definition}[Lattice lines]
  \label{def:lline}
  For a primitive direction \(d=(p,q)\) and \(c\in\Z\), the \emph{lattice line} \(\ell_{d,c}\) is the
  level set
  \begin{equation}
    \label{eq:lline-def}
    \ell_{d,c} \;=\; \varphi_d^{-1}(c) \;=\; \{(x,y)\in\Z^2 : qx-py=c\}.
  \end{equation}
  Every line in \(\R^2\) containing at least two points of \(\Z^2\) equals \(\ell_{d,c}\) for a unique
  primitive direction \(d\) (up to sign) and a unique \(c\in\Z\); we call such a line a lattice line of
  direction \(d\).
\end{definition}

\begin{definition}[Covering family, valid family]
  \label{def:family}
  A \emph{covering family} is a set \(\mathcal F\) of lattice lines with \(\bigcup\mathcal F\supseteq\Z^2\). A
  covering family \(\mathcal F\) is \emph{valid} if no two lines of \(\mathcal F\) of different
  direction share a point of \(\Z^2\); equivalently, for every pair of distinct primitive directions
  \(d,d'\) occurring in \(\mathcal F\), and every \(\ell_{d,c},\ell_{d',c'}\in\mathcal F\), we have
  \(\ell_{d,c}\cap\ell_{d',c'}\cap\Z^2=\emptyset\).
\end{definition}

\begin{figure}[h]
  \centering
  \begin{tikzpicture}[scale=0.9]
    \foreach \x in {-1,...,4} \draw[gray!40,thin] (\x,-1) -- (\x,4);
    \foreach \y in {-1,...,4} \draw[gray!40,thin] (-1,\y) -- (4,\y);
    \foreach \x in {-1,...,4} \foreach \y in {-1,...,4} \fill[black] (\x,\y) circle (0.10);

    \definecolor{lineAcolor}{RGB}{0,90,181}
    \definecolor{lineBcolor}{RGB}{190,30,30}

    \draw[lineAcolor, line width=0.85pt] (-0.6,0.7) -- (3.6,2.8);
    \foreach \p in {(0,1),(2,2),(4,3)} \fill[lineAcolor] \p circle (2.4pt);

    \draw[lineBcolor, line width=0.85pt] (0.2,4.6) -- (3.3,-1.6);
    \foreach \p in {(1,3),(2,1)} \fill[lineBcolor] \p circle (2.4pt);

    \fill[black, rotate around={45:(1.6,1.8)}] (1.6,1.8) ++(-2.6pt,-2.6pt) rectangle ++(5.2pt,5.2pt);
    \node[below right, font=\scriptsize] at (1.6,1.8) {\(\ell_{d,c}\cap\ell_{d',c'}\notin\Z^2\)};
  \end{tikzpicture}
  \caption{Two lattice lines of different directions \(d,d'\) may cross \emph{off} the lattice
  (diamond marker) without violating validity (Definition~\ref{def:family}): only crossings
  \emph{at} points of \(\Z^2\) (dots) are forbidden between distinct directions.}
  \label{fig:offlattice-crossing}
\end{figure}
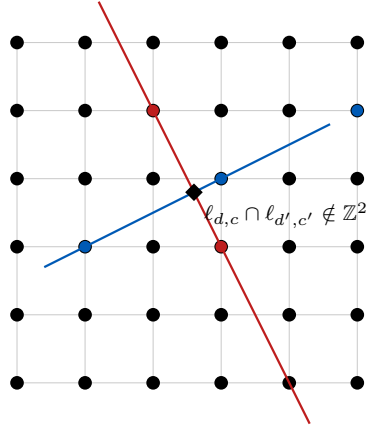

\begin{lemma}[The everywhere-disjoint variant is trivial]
  \label{lem:everywhere-disjoint-trivial}
  If the lines of \(\mathcal F\) are required to be pairwise disjoint everywhere (not merely off the
  lattice), then every line in \(\mathcal F\) shares a single direction, and any one direction
  already suffices (Example~\ref{ex:single-direction} below).
\end{lemma}

\begin{proof}
  Two distinct lines in \(\R^2\) are either parallel or meet at exactly one point, so pairwise
  disjointness forces every line in the family to share a single direction.
\end{proof}

All the content is in the ``crossings allowed off the lattice'' variant of
Definition~\ref{def:family}, which is our subject here.

\begin{example}
  \label{ex:single-direction}
  For a primitive direction \(d=(p,q)\), the lines \(\ell_{d,c}\), \(c\in\Z\), partition \(\Z^2\) (they are
  the fibers of the surjective homomorphism \(\varphi_d\)):
  \begin{equation}
    \label{eq:single-direction-partition}
    \Z^2 \;=\; \bigsqcup_{c\in\Z} \ell_{d,c}.
  \end{equation}
  So a single direction always suffices to cover \(\Z^2\) by pairwise-disjoint (in particular
  lattice-point-disjoint) lattice lines.
\end{example}

Our main result:

\begin{theorem}
  \label{thm:main}
  There is a valid covering family \(\mathcal F\) (Definition~\ref{def:family}) whose set of realized
  directions is dense in \(\RP\).
\end{theorem}

\begin{figure}[h]
  \centering
  \includegraphics[width=0.55\textwidth]{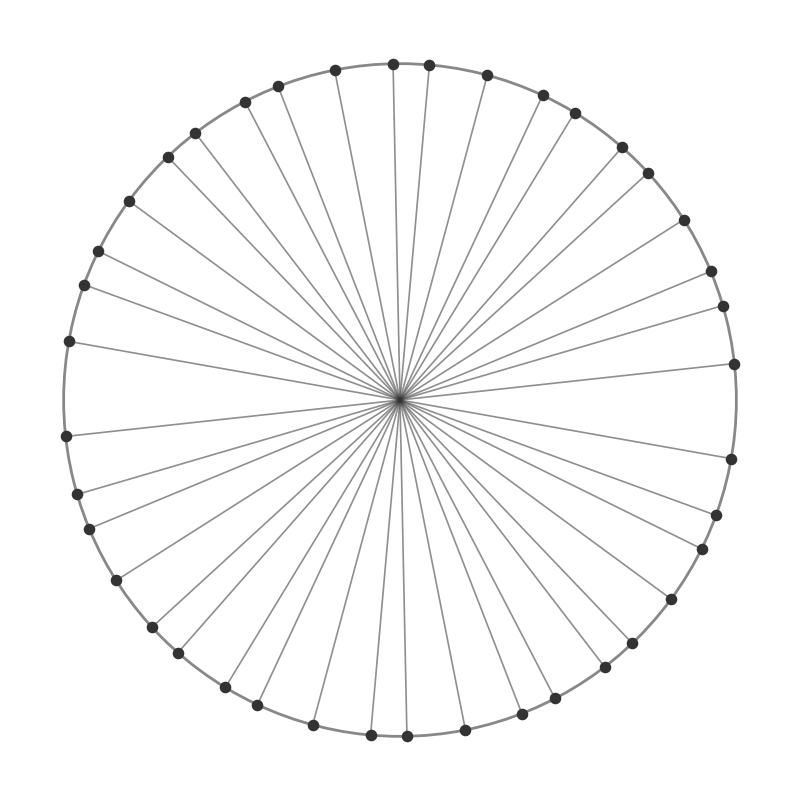}
  \caption{A preview of Theorem~\ref{thm:main}: the first \(100\) directions realized by the
  construction of Section~\ref{sec:construction} below, each drawn as a diameter of the unit
  circle. See Figure~\ref{fig:directions} for the same data with the construction's progress
  (order and consecutive jumps) made visible.}
  \label{fig:directions-preview}
\end{figure}

\section{Two lemmas on lattice lines}

\begin{lemma}[Rigidity]
  \label{lem:rigidity}
  Let \(d_1=(p_1,q_1)\), \(d_2=(p_2,q_2)\) be \emph{two} distinct primitive directions, and set
  \begin{equation}
    \label{eq:Delta-def}
    \Delta \;=\; p_1q_2-p_2q_1 \;\ne\; 0.
  \end{equation}
  The lines \(\ell_{d_1,c_1}\) and \(\ell_{d_2,c_2}\) meet at a lattice point if and only if
  \begin{equation}
    \label{eq:rigidity-conclusion}
    \tag{R}
    \Delta \mid (p_1c_2-p_2c_1) \qquad\text{and}\qquad \Delta \mid (q_1c_2-q_2c_1).
  \end{equation}
  In particular, if \(|\Delta|=1\), condition~\eqref{eq:rigidity-conclusion} holds vacuously for
  \emph{every} choice of \(c_1,c_2\), so \(d_1,d_2\) can never both occur in a valid family
  (Definition~\ref{def:family}).
\end{lemma}

\begin{proof}
  A lattice point \((x,y)\) lies on both lines exactly when the linear system
  \begin{equation}
    \label{eq:rigidity-system}
    \begin{aligned}
      q_1x - p_1y &= c_1, \\
      q_2x - p_2y &= c_2
    \end{aligned}
  \end{equation}
  holds. Eliminating \(y\) gives
  \begin{equation}
    \label{eq:rigidity-elim-x}
    \Delta\, x \;=\; p_1c_2 - p_2c_1,
  \end{equation}
  and eliminating \(x\) symmetrically gives
  \begin{equation}
    \label{eq:rigidity-elim-y}
    \Delta\, y \;=\; q_1c_2 - q_2c_1.
  \end{equation}
  Since \(\Delta\ne0\), equations~\eqref{eq:rigidity-elim-x}--\eqref{eq:rigidity-elim-y} determine
  \(x,y\in\mathbb Q\) uniquely (this is exactly Cramer's rule applied to~\eqref{eq:rigidity-system}),
  and \((x,y)\in\Z^2\) if and only if \(\Delta\) divides both right-hand sides, which is precisely
  \eqref{eq:rigidity-conclusion}. Conversely, if~\eqref{eq:rigidity-conclusion} holds, the
  quotients in~\eqref{eq:rigidity-elim-x}--\eqref{eq:rigidity-elim-y} are integers and one checks
  directly, by substituting back into~\eqref{eq:rigidity-system}, that they solve the system. If
  \(|\Delta|=1\), both
  divisibility conditions in~\eqref{eq:rigidity-conclusion} hold automatically for every
  \(c_1,c_2\in\Z\).
\end{proof}

\begin{lemma}[Coset]
  \label{lem:coset}
  Let \(p,q\) be coprime integers, \(x_0,y_0\in\Z\), and set
  \begin{equation}
    \label{eq:c0-def}
    c_0 = qx_0-py_0.
  \end{equation}
  For \(k\in\Z\), write \(r_k=c_0+pqk\). Then
  \begin{equation}
    \label{eq:coset-conclusion}
    \bigcup_{k\in\Z} \ell_{(p,q),\,r_k} \;=\; \{(x,y)\in\Z^2 : x\equiv x_0\!\!\pmod{|p|},\
    y\equiv y_0\!\!\pmod{|q|}\}.
  \end{equation}
\end{lemma}

\begin{proof}
  Write \(C\) for the right-hand side of~\eqref{eq:coset-conclusion}. Since \(\gcd(p,q)=1\) implies in
  particular \(p,q\ne0\), congruence modulo \(p\) (resp.\ \(q\)) is well defined, and \(C\) admits the
  parametrization
  \begin{equation}
    \label{eq:coset-param}
    C = \{(x_0+pu,\,y_0+qv) : u,v\in\Z\}.
  \end{equation}
  For a point of the form~\eqref{eq:coset-param},
  \begin{equation}
    \label{eq:coset-phi-eval}
    \varphi_{(p,q)}(x_0+pu,\,y_0+qv) \;=\; q(x_0+pu)-p(y_0+qv) \;=\; c_0 + pq(u-v),
  \end{equation}
  using~\eqref{eq:c0-def} and cancelling the \(pqu\) terms. By~\eqref{eq:lline-def}, this point lies
  on \(\ell_{(p,q),r_k}\) if and only if \(u-v=k\). As \(u\) ranges over \(\Z\) (with \(v=u-k\) then
  determined), every point of \(C\) with that particular value of \(u-v\) is visited exactly once, and
  \eqref{eq:coset-phi-eval} shows every point of \(\ell_{(p,q),r_k}\cap C\) has \(u-v=k\): hence
  \begin{equation}
    \label{eq:coset-single-level}
    \ell_{(p,q),\,r_k} \cap C \;=\; \{(x_0+pu,\,y_0+q(u-k)) : u\in\Z\},
  \end{equation}
  and moreover every point of \(\ell_{(p,q),r_k}\) manifestly lies in \(C\) (take \(u,v\) with
  \(u-v=k\) so that~\eqref{eq:coset-phi-eval} matches \(r_k\); such \((x,y)\) automatically satisfies
  \(x\equiv x_0\pmod p\), \(y\equiv y_0\pmod q\)). So \(\ell_{(p,q),r_k}\subseteq C\) for every \(k\),
  with equality of the union to all of \(C\) once \(k\) ranges over \(\Z\), since every pair \((u,v)\in\Z^2\)
  has \emph{some} value of \(u-v\).
\end{proof}

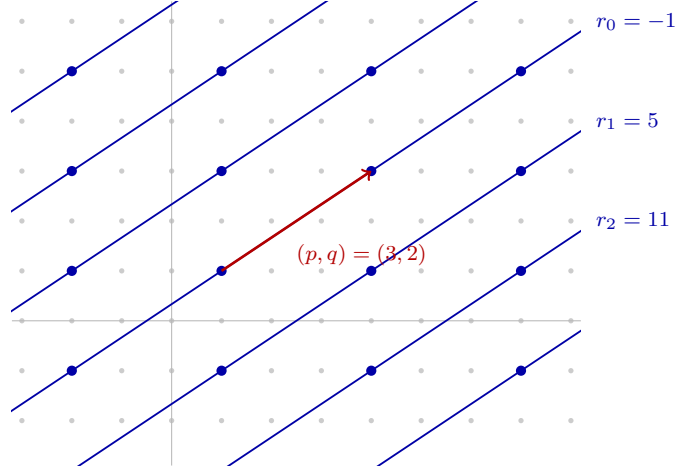
\begin{figure}[htbp]
  \centering
  \begin{tikzpicture}[scale=0.66]
    \begin{scope}
      \clip (-3.2,-2.9) rectangle (8.2,6.4);
      \foreach \x in {-3,...,8} \foreach \y in {-3,...,7}
        \fill[gray!40] (\x,\y) circle (1.5pt);
      \draw[gray!55] (-3.2,0) -- (8.2,0);
      \draw[gray!55] (0,-2.9) -- (0,6.4);
      \foreach \r in {-19,-13,-7,-1,5,11,17}
        \draw[blue!65!black, line width=0.7pt] (-8,{(2*(-8)-\r)/3}) -- (14,{(2*14-\r)/3});
      \foreach \x in {-2,1,4,7} \foreach \y in {-1,1,3,5}
        \fill[blue!65!black] (\x,\y) circle (2.9pt);
      \draw[->, red!70!black, line width=1pt] (1,1) -- (4,3);
    \end{scope}
    \node[red!70!black, font=\scriptsize, below right] at (2.3,1.75) {$(p,q)=(3,2)$};
    \node[blue!65!black, font=\scriptsize, right] at (8.3,6.0) {$r_0=-1$};
    \node[blue!65!black, font=\scriptsize, right] at (8.3,4.0) {$r_1=5$};
    \node[blue!65!black, font=\scriptsize, right] at (8.3,2.0) {$r_2=11$};
  \end{tikzpicture}
  \caption{Lemma~\ref{lem:coset} for \(d=(p,q)=(3,2)\) and \((x_0,y_0)=(1,1)\), so that
    \(c_0=qx_0-py_0=-1\) and \(r_k=-1+6k\). Small grey dots are \(\Z^2\); large dots are the coset
    \(C=\{(x,y): x\equiv 1 \bmod 3,\ y\equiv 1 \bmod 2\}\), a translate of the sublattice
    \(3\Z\times2\Z\). Each line \(\ell_{(3,2),r_k}\) meets \(\Z^2\) \emph{only} in \(C\), and
    consecutive points along a line differ by \((p,q)=(3,2)\); as \(k\) ranges over \(\Z\) the
    lines partition \(C\). The levels are equally spaced with gap \(pq=6\): a smaller gap would
    give lines missing \(\Z^2\) altogether, a larger one would leave points of \(C\) uncovered.}
  \label{fig:coset}
\end{figure}

Recall that by \emph{Farey neighbors} is meant two primitive directions
\[d_1=(p_1,q_1), \quad d_2=(p_2,q_2)\]
(equivalently, fractions \(p_1/q_1\), \(p_2/q_2\) in lowest
terms) with
\[|p_1q_2-p_2q_1|=1.\]
This is exactly the \(|\Delta|=1\) case of
Lemma~\ref{lem:rigidity}; see \cite[Ch.~III]{HardyWright2008} for the classical theory.
The rigidity lemma shows that a family realizing many directions cannot simply throw them
together: directions related by a determinant of \(\pm1\) (``unimodular pairs'', e.g.\ any pair of
Farey neighbors read as slopes) are permanently incompatible. The coset lemma is the tool that lets
us build a family avoiding all bad pairs at once: applied recursively below, it splits the current
sub-coset into finer sub-cosets, one of which is handed off to each freshly chosen direction.

\section{The construction}
\label{sec:construction}

\begin{definition}[Rank-2 coset]
  \label{def:Rcoset}
  For coprime positive integers \(n_1,n_2\) and \(x_0,y_0\in\Z\), define
  \begin{equation}
    \label{eq:Rcoset-def}
    R(n_1,n_2,x_0,y_0) \;=\; \{(x,y)\in\Z^2 : x\equiv x_0\!\!\pmod{n_1},\ y\equiv y_0\!\!\pmod{n_2}\}.
  \end{equation}
  Since \(n_1,n_2>0\), every \((x,y)\in R(n_1,n_2,x_0,y_0)\) is of the form \(x=x_0+n_1u\), \(y=y_0+n_2w\)
  for a \emph{unique} pair \((u,w)\in\Z^2\); call \((u,w)\) the \emph{internal coordinates} of the
  point.
\end{definition}

We state the two halves of the original Splitting lemma as two separate lemmas: the first
identifies each finer sub-coset with a plain coset of a new direction, and the second records that
these sub-cosets genuinely partition \(R\).

\begin{lemma}[Splitting: coset identity]
  \label{lem:splitting-coset}
  Let \(R=R(n_1,n_2,x_0,y_0)\), let \(s,t\) be nonzero integers with
  \begin{equation}
    \label{eq:splitting-hyps}
    \tag{H}
    \gcd(s,n_2)=\gcd(t,n_1)=\gcd(s,t)=1,
  \end{equation}
  and put
  \begin{equation}
    \label{eq:PQ-def}
    P=n_1s, \qquad Q=n_2t.
  \end{equation}
  For \(u_0,w_0\in\Z\), let
  \begin{equation}
    \label{eq:Rsub-def}
    \begin{aligned}
      R_{u_0,w_0} \;=\; \{(x,y)\in R :\ & x=x_0+n_1u,\ y=y_0+n_2w \text{ for some } \\
      & u\equiv u_0\!\!\pmod{|s|},\ w\equiv w_0\!\!\pmod{|t|}\}
    \end{aligned}
  \end{equation}
  (a sub-coset of \(R\), picked out via the internal coordinates of Definition~\ref{def:Rcoset}).
  Then
  \begin{equation}
    \label{eq:splitting-coset-conclusion}
    \tag{S}
    \bigcup_{k\in\Z} \ell_{(P,Q),\,c_0+PQk} \;=\; R_{u_0,w_0}, \qquad c_0 =
    Q(x_0+n_1u_0)-P(y_0+n_2w_0).
  \end{equation}
  In particular \((P,Q)\) is itself a primitive direction: \(\gcd(P,Q)=1\).
\end{lemma}

\begin{proof}
  \emph{Primitivity of \((P,Q)\).} A prime dividing both \(P=n_1s\) and \(Q=n_2t\) would have to divide
  one of \(n_1,n_2\) (excluded, as \(\gcd(n_1,n_2)=1\)), or \(n_1\) and \(t\) (excluded by \(\gcd(t,n_1)=1\)
  in~\eqref{eq:splitting-hyps}), or \(s\) and \(n_2\) (excluded by \(\gcd(s,n_2)=1\)), or \(s\) and \(t\)
  (excluded by \(\gcd(s,t)=1\)). These four cases are exhaustive, so no such prime exists.

  \emph{The identity~\eqref{eq:splitting-coset-conclusion}.} For \((x,y)=(x_0+n_1u,\,y_0+n_2w)\),
  \begin{equation}
    \label{eq:splitting-phi-eval}
    Qx-Py \;=\; Q(x_0+n_1u) - P(y_0+n_2w) \;=\; (Qx_0-Py_0) + n_1n_2(tu-sw),
  \end{equation}
  using \(Q=n_2t\), \(P=n_1s\). Put \(c_0^{\mathrm{loc}}=tu_0-sw_0\). Then, by Lemma~\ref{lem:coset}
  applied to the coprime pair \((s,t)\) in the \emph{internal} coordinates \((u,w)\),
  \begin{equation}
    \label{eq:splitting-local-coset}
    \bigcup_{k\in\Z}\{(u,w) : tu-sw = c_0^{\mathrm{loc}}+stk\} \;=\; \{u\equiv
    u_0\!\!\pmod{|s|},\ w\equiv w_0\!\!\pmod{|t|}\}.
  \end{equation}
  Substituting the local level \(c_0^{\mathrm{loc}}+stk\) into~\eqref{eq:splitting-phi-eval} via
  \(x_0,y_0\) translates it into the global level
  \begin{equation}
    \label{eq:splitting-level-translate}
    \begin{aligned}
      (Qx_0-Py_0) + n_1n_2\big(c_0^{\mathrm{loc}}+stk\big) \;&=\; c_0 + n_1n_2st\,k \\
      \;&=\; c_0+PQk,
    \end{aligned}
  \end{equation}
  where the first equality uses \(c_0=(Qx_0-Py_0)+n_1n_2c_0^{\mathrm{loc}}\) (a direct computation
  confirming this matches the formula for \(c_0\) stated in~\eqref{eq:splitting-coset-conclusion}),
  and the second uses \(PQ=n_1n_2st\). Combining~\eqref{eq:splitting-local-coset} and
  \eqref{eq:splitting-level-translate} with~\eqref{eq:splitting-phi-eval} gives exactly
  \eqref{eq:splitting-coset-conclusion}.
\end{proof}

\begin{example}
  \label{ex:splitting-worked}
  Take \(n_1=2\), \(n_2=3\), \(x_0=y_0=1\), so
  \[
    R=R(2,3,1,1)=\{(x,y):x\text{ odd},\ y\equiv1\!\!\pmod3\}.
  \]
  Take \(s=2\), \(t=3\); one checks directly that~\eqref{eq:splitting-hyps} holds:
  \begin{align}
    \gcd(s,n_2) &= \gcd(2,3) = 1, \\
    \gcd(t,n_1) &= \gcd(3,2) = 1, \\
    \gcd(s,t) &= \gcd(2,3) = 1.
  \end{align}
  Then \(P=n_1s=4\), \(Q=n_2t=9\) (and indeed \(\gcd(4,9)=1\)). Take the residue pair \((u_0,w_0)=(1,2)\);
  then
  \begin{equation}
    c_0 = Q(x_0+n_1u_0)-P(y_0+n_2w_0) = 9\cdot(1+2)-4\cdot(1+6) = 27-28 = -1.
  \end{equation}
  Three points of \(R_{1,2}\), at three different values of \(k\) in
  \eqref{eq:splitting-coset-conclusion}:
  \begin{align}
    (u,w)&=(1,2) \;\longmapsto\; (x,y)=(3,7): \notag\\
    &\varphi_{(4,9)}(3,7)=27-28=-1=c_0+36\cdot0, \\
    (u,w)&=(3,2) \;\longmapsto\; (x,y)=(7,7): \notag\\
    &\varphi_{(4,9)}(7,7)=63-28=35=c_0+36\cdot1, \\
    (u,w)&=(1,5) \;\longmapsto\; (x,y)=(3,16): \notag\\
    &\varphi_{(4,9)}(3,16)=27-64=-37=c_0+36\cdot(-1),
  \end{align}
  using \(PQ=36\) throughout.
\end{example}

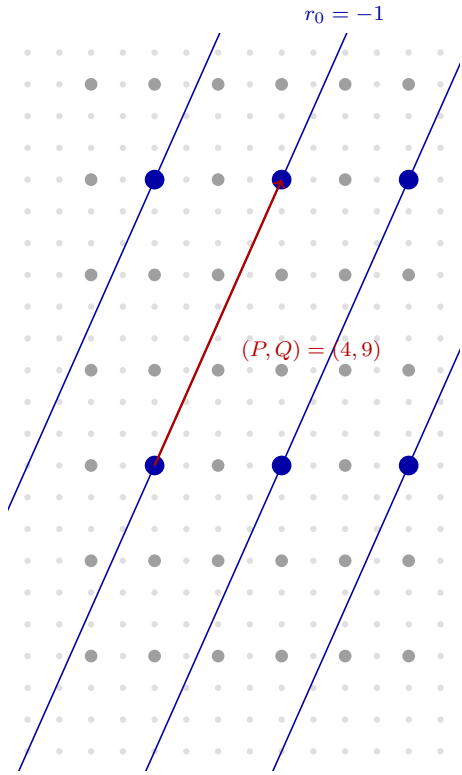
\begin{figure}[htbp]
  \centering
  \begin{tikzpicture}[scale=0.42]
    \begin{scope}
      \clip (-1.6,-2.6) rectangle (12.6,20.6);
      \foreach \x in {-1,...,12} \foreach \y in {-2,...,20}
        \fill[gray!25] (\x,\y) circle (0.10);
      \foreach \x in {1,3,...,11} \foreach \y in {1,4,...,19}
        \fill[gray!75] (\x,\y) circle (0.20);
      \foreach \r in {-37,-1,35,71}
        \draw[blue!65!black, line width=0.6pt] ({(\r+4*(-6))/9},-6) -- ({(\r+4*26)/9},26);
      \foreach \x in {3,7,11} \foreach \y in {7,16}
        \fill[blue!65!black] (\x,\y) circle (0.31);
      \draw[->, red!70!black, line width=0.9pt] (3,7) -- (7,16);
    \end{scope}
    \node[red!70!black, font=\scriptsize, right] at (5.4,10.6) {$(P,Q)=(4,9)$};
    \node[blue!65!black, font=\scriptsize, above] at (8.98,20.6) {$r_0=-1$};
  \end{tikzpicture}
  \caption{Example~\ref{ex:splitting-worked}. Small grey dots are \(\Z^2\); medium grey dots the
    coarse coset \(R=R(2,3,1,1)=\{x \text{ odd},\ y\equiv1\bmod 3\}\); large dots the fine
    sub-coset \(R_{1,2}=\{x\equiv3\bmod4,\ y\equiv7\bmod9\}\) singled out by the residue pair
    \((u_0,w_0)=(1,2)\). By Lemma~\ref{lem:splitting-coset} that sub-coset is again a plain coset,
    now for the direction \((P,Q)=(4,9)\), and Lemma~\ref{lem:coset} applies to it verbatim: the
    lines shown are \(\ell_{(4,9),r_k}\) with \(r_k=-1+36k\). Consecutive points along a line
    differ by \((4,9)\); the arrow marks \((3,7)\mapsto(7,16)\), two of the three points computed
    in the example. Note that the lines pass through \emph{no} point of \(R\) outside
    \(R_{1,2}\) --- that confinement is the content of the lemma, and what makes the recursion
    safe.}
  \label{fig:splitting-example}
\end{figure}

\begin{lemma}[Splitting: partition]
  \label{lem:splitting-partition}
  With \(R\), \(s\), \(t\) as in Lemma~\ref{lem:splitting-coset},
  \begin{equation}
    \label{eq:splitting-partition-conclusion}
    R \;=\; \bigsqcup_{\substack{0\le u_0<|s| \\ 0\le w_0<|t|}} R_{u_0,w_0}
  \end{equation}
  (a disjoint union over the \(|s||t|\) residue pairs).
\end{lemma}

\begin{proof}
  Fix \((x,y)\in R\), with internal coordinates \((u,w)\) as in Definition~\ref{def:Rcoset}. By the
  division algorithm,
  \[
    u=s\Big\lfloor\frac us\Big\rfloor+u_0, \qquad 0\le u_0<|s|,
  \]
  for a unique such \(u_0\) (\(u_0\) is the unique representative of \(u\) modulo \(s\) lying in
  \([0,|s|)\), and similarly \(w_0\in[0,|t|)\) for \(w\) modulo \(t\)). By construction,
  \[
    u\equiv u_0\!\!\pmod{|s|} \quad\text{and}\quad w\equiv w_0\!\!\pmod{|t|},
  \]
  so \((x,y)\in R_{u_0,w_0}\) for this particular pair
  \((u_0,w_0)\), establishing
  \[
    R\;\subseteq\;\bigcup_{u_0,w_0}R_{u_0,w_0};
  \]
  the reverse inclusion is immediate from~\eqref{eq:Rsub-def}. For disjointness (equivalently,
  uniqueness of \((u_0,w_0)\) for each point): suppose
  \[
    (x,y)\in R_{u_0,w_0}\cap R_{u_0',w_0'}
  \]
  via internal coordinates \((u,w)\) and \((u'',w'')\) respectively. Since internal coordinates are unique
  (Definition~\ref{def:Rcoset}), \(u''=u\) and \(w''=w\), so
  \[
    u_0\equiv u\equiv u_0'\!\!\pmod{|s|};
  \]
  as both
  \(u_0,u_0'\in[0,|s|)\), an interval of length \(|s|\), and their difference is a multiple of \(s\) of
  absolute value \(<|s|\), we get \(u_0=u_0'\), and symmetrically \(w_0=w_0'\).
\end{proof}

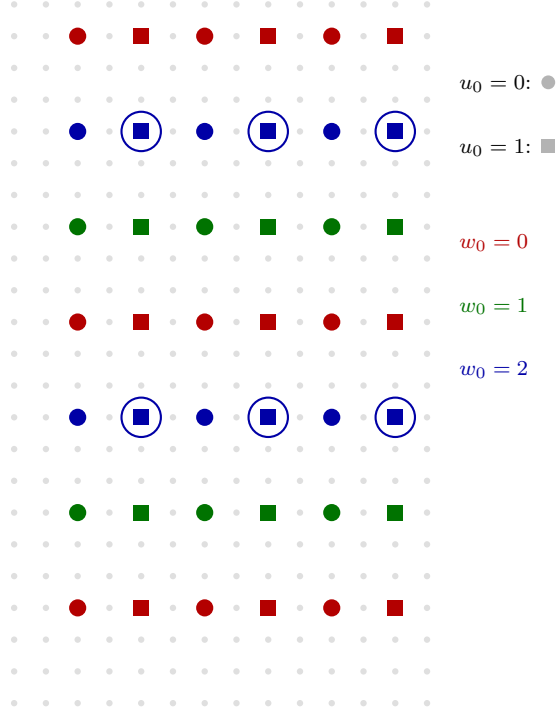
\begin{figure}[htbp]
  \centering
  \begin{tikzpicture}[scale=0.42]
    \begin{scope}
      \clip (-1.6,-2.6) rectangle (12.6,20.6);
      \foreach \x in {-1,...,12} \foreach \y in {-2,...,20}
        \fill[gray!25] (\x,\y) circle (0.10);
      \foreach \x in {1,5,9} {
        \foreach \y in {1,10,19} \fill[red!70!black]    (\x,\y) circle (0.27);
        \foreach \y in {4,13}    \fill[green!45!black]  (\x,\y) circle (0.27);
        \foreach \y in {7,16}    \fill[blue!65!black]   (\x,\y) circle (0.27);
      }
      \foreach \x in {3,7,11} {
        \foreach \y in {1,10,19} \fill[red!70!black]    (\x-0.25,\y-0.25) rectangle (\x+0.25,\y+0.25);
        \foreach \y in {4,13}    \fill[green!45!black]  (\x-0.25,\y-0.25) rectangle (\x+0.25,\y+0.25);
        \foreach \y in {7,16}    \fill[blue!65!black]   (\x-0.25,\y-0.25) rectangle (\x+0.25,\y+0.25);
      }
      \foreach \x in {3,7,11} \foreach \y in {7,16}
        \draw[blue!65!black, line width=0.8pt] (\x,\y) circle (0.62);
    \end{scope}
    \node[font=\scriptsize, right] at (12.7,17.5) {$u_0=0$: \tikz\fill[gray!60] (0,0) circle (2.6pt);};
    \node[font=\scriptsize, right] at (12.7,15.5) {$u_0=1$: \tikz\fill[gray!60] (-0.09,-0.09) rectangle (0.09,0.09);};
    \node[font=\scriptsize, right, red!70!black]   at (12.7,12.5) {$w_0=0$};
    \node[font=\scriptsize, right, green!45!black] at (12.7,10.5) {$w_0=1$};
    \node[font=\scriptsize, right, blue!65!black]  at (12.7,8.5)  {$w_0=2$};
  \end{tikzpicture}
  \caption{Lemma~\ref{lem:splitting-partition} for the \(R\), \(s=2\), \(t=3\) of
    Example~\ref{ex:splitting-worked}. Every point of the coarse coset \(R\) carries exactly one of
    the \(|s||t|=6\) residue pairs \((u_0,w_0)\): shape records \(u_0\in\{0,1\}\), colour records
    \(w_0\in\{0,1,2\}\). The six classes are the sub-cosets \(R_{u_0,w_0}\), and the picture is the
    lemma: no point carries two labels and none carries none, so the classes are disjoint and
    exhaust \(R\). Circled is the class \(R_{1,2}\) of
    Figure~\ref{fig:splitting-example} --- one of the six, handed to a single new direction, while
    the other five remain available to later steps of the recursion.}
  \label{fig:splitting-partition}
\end{figure}

A naive approach to this construction might attempt to realize an \emph{arbitrary} target
direction by using a multiple of it known to lie in the current sublattice, without checking that
the multiple's own geometric line (which follows its \emph{primitive} direction, not the multiple
used to construct it) stays confined to the intended coset; this is exactly the pitfall
Lemma~\ref{lem:splitting-coset} avoids, via the primitivity argument built into its own proof.

\begin{lemma}[Steering]
  \label{lem:steering}
  Let \(n_1,n_2\) be coprime positive integers. For every \(\theta\in[0,\pi)\) and every
  \(\varepsilon>0\) there exist nonzero integers \(s,t\), \(|s|,|t|\ge2\), satisfying
  \eqref{eq:splitting-hyps}, such that the direction of \((n_1s,n_2t)\) lies within \(\varepsilon\) of
  \(\theta\).
\end{lemma}

\begin{proof}
  By symmetry (swap the roles of \(s\) and \(t\) to treat \(\theta=\pi/2\), which targets reciprocal
  slope \(0\) and is handled by the same argument below with the roles of \(n_1,n_2\) exchanged,
  subject to the same window shift introduced at~\eqref{eq:steering-t0}) it suffices to treat
  \(\theta\ne\pi/2\), i.e.\ a finite target slope \(\mu=\tan\theta\); the case \(\mu<0\) is identical
  with \(t\mapsto-t\), so assume \(\mu\ge0\).

  Set
  \begin{equation}
    \label{eq:steering-constants}
    \delta \;=\; \frac12\prod_{r\mid n_1}\Big(1-\frac1r\Big) \;>\; 0, \qquad E \;=\;
    2^{\omega(n_1)+1},
  \end{equation}
  both depending only on \(n_1\), where \(\omega(n_1)\) is the number of distinct prime divisors of
  \(n_1\). For a prime \(p\nmid n_1n_2\) and any interval \(I\subset\Z\) of length \(L\),
  inclusion--exclusion over the divisors of \(n_1p\) gives
  \begin{align}
    \#\{t\in I : \gcd(t,n_1p)=1\}
      &\;\ge\; L\prod_{r\mid n_1}\Big(1-\frac1r\Big)\Big(1-\frac1p\Big) - 2^{\omega(n_1)+1} \notag\\
      &\;\ge\; \delta L - E, \label{eq:steering-sieve}
  \end{align}
  using \(p\ge2\) for the last factor. Fix
  \begin{equation}
    \label{eq:steering-L0}
    L_0 = \Big\lceil \frac{E+1}{\delta} \Big\rceil,
  \end{equation}
  independent of \(p\) --- and, crucially, independent of \(\mu\) as well. Choose a prime \(p\nmid n_1n_2\)
  large enough that
  \begin{equation}
    \label{eq:steering-p-choice}
    L_0 + 2 \;<\; \frac{\varepsilon n_1p}{n_2}
  \end{equation}
  This is possible since the right side \(\to\infty\) as \(p\to\infty\), for fixed
  \(\varepsilon,n_1,n_2\). Note that \eqref{eq:steering-p-choice} does not involve \(\mu\) at all,
  unlike \(t_0\) below. Set \(s=p\) and
  \begin{equation}
    \label{eq:steering-t0}
    t_0 = \Big\lfloor \frac{\mu n_1p}{n_2} \Big\rfloor + 2.
  \end{equation}
  The \(+2\) shift in~\eqref{eq:steering-t0} is what
  makes the forthcoming \(|t|\ge2\) requirement hold \emph{unconditionally} rather than only "for \(p\)
  large": since \(\mu\ge0\), already \(t_0\ge2\) by~\eqref{eq:steering-t0}, so every integer in the
  window \(I\) below inherits \(|t|\ge2\) directly, with no degenerate behaviour at \(\mu=0\) (where,
  without the shift, \(t_0=0\) for \emph{every} \(p\), the window never moves away from the origin as
  \(p\to\infty\), and \(t=1\), always coprime to \(n_1p\), is a witness that nothing rules out being
  the \emph{only} one, violating \(|t|\ge2\)). The interval \(I=[t_0,t_0+L_0)\) then contains, by
  \eqref{eq:steering-sieve} with \(L=L_0\) and~\eqref{eq:steering-L0}, at least \(\delta L_0-E\ge1\)
  integer \(t\) with \(\gcd(t,n_1p)=1\) --- equivalently \(\gcd(t,n_1)=\gcd(t,s)=1\), as \(s=p\) is prime
  --- and, as just noted, \(|t|\ge2\) automatically. For any such \(t\),
  \begin{equation}
    \label{eq:steering-final-bound}
    \Big|\frac{n_2t}{n_1s}-\mu\Big| \;=\; \frac{n_2}{n_1p}\Big|t-\frac{\mu n_1p}{n_2}\Big| \;<\;
    \frac{n_2(L_0+2)}{n_1p} \;<\; \varepsilon,
  \end{equation}
  using~\eqref{eq:steering-p-choice} for the last step and \(\big|t-\mu n_1p/n_2\big|<L_0+2\), which
  follows from \(t\in[t_0,t_0+L_0)\) and \(t_0-\mu n_1p/n_2\in(1,2]\) by~\eqref{eq:steering-t0}, so
  \((P,Q)=(n_1s,n_2t)\) has slope within \(\varepsilon\) of \(\mu\), hence (as \(\arctan\) is
  \(1\)-Lipschitz, and \(\arctan\mu=\theta-\pi\) when \(\theta>\pi/2\) identifies the same point of
  \(\RP\)) direction within \(\varepsilon\) of \(\theta\).
\end{proof}

The point of Lemma~\ref{lem:steering} is that it is \emph{not} enough to know that \emph{some}
admissible \((s,t)\) exists at every stage: repeatedly reusing, say, \((s,t)=(2,3)\) is admissible
forever (all three coprimality conditions of~\eqref{eq:splitting-hyps} hold trivially against the
resulting \(n_1=2^{k},n_2=3^{k}\)), yet gives directions converging to a single point of \(\RP\), not a
dense set. What is needed, and what Lemma~\ref{lem:steering} provides, is the ability to steer
toward an \emph{arbitrary prescribed} target at every stage.

\begin{remark}[Tails of a dense sequence remain dense]
  \label{rem:tail-density}
  If \(X\) is a \(T_1\) topological space with no isolated points, \(D\subseteq X\) is dense, and
  \(F\subset X\) is finite, then \(D\setminus F\) is still dense in \(X\). This follows since any
  nonempty open \(U\) is itself infinite, as no point of \(X\) is isolated; so \(U\setminus F\) is
  nonempty and open, hence meets \(D\), giving a point of \(D\setminus F\) in \(U\). If we apply this
  with \(X=\RP\) --- compact, metric, hence \(T_1\), and with no isolated points, since it is a
  circle --- and \(D=\{\theta_k\}_{k\ge1}\), then every tail \(\{\theta_k\}_{k\ge N}\) is dense too,
  since it contains \(D\setminus\{\theta_1,\dots,\theta_{N-1}\}\).
  This is the fact the density argument below needs beyond mere density of \(\{\theta_k\}\) itself
  --- without ``no isolated points'' it can fail (e.g.\ a sequence dense in \(\{0\}\cup[1,2]\) that
  visits \(0\) only once).
\end{remark}

We now have all pieces needed to prove Theorem~\ref{thm:main}.

\begin{proof}[Proof of Theorem~\ref{thm:main}]
  Fix an enumeration of \(\Z^2\), \((z_k)_{k\ge1}\), and a sequence \((\theta_k)_{k\ge1}\) dense in
  \([0,\pi)\) (e.g.\ \(\theta_k=k\alpha\bmod\pi\) for \(\alpha/\pi\) irrational, by equidistribution of
  the irrational rotation). Set \(R_0=\Z^2\), i.e.\ \((n_1,n_2,x_0,y_0)=(1,1,0,0)\).

  Given \(R_{k-1}=R(n_1,n_2,x_0,y_0)\), apply Lemma~\ref{lem:steering} to find \(s,t\) realizing a
  direction within \(1/k\) of \(\theta_k\) --- \(s,t\) may come out negative, and
  Lemma~\ref{lem:steering} allows this, since both signs of slope need to be reachable. Then apply
  Lemmas~\ref{lem:splitting-coset}--\ref{lem:splitting-partition} to split \(R_{k-1}\) into its
  \(|s||t|\ge4\) sub-cosets, each fully coverable by direction \((P_k,Q_k)=(n_1s,n_2t)\) restricted
  to one residue class of levels. Put into \(\mathcal F\) all such lines for every sub-coset except
  one: if \(z_k\in R_{k-1}\), reserve the lexicographically-least \((u_0,w_0)\) whose sub-coset does
  not contain \(z_k\), which is possible since there are \(\ge4\) sub-cosets and only one need be
  avoided; otherwise reserve the lexicographically-least \((u_0,w_0)\) outright. Let
  \(R_k=R_{u_0,w_0}\) be that reserved sub-coset; Figure~\ref{fig:nesting} shows the first two
  stages of the resulting chain \(\Z^2=R_0\supset R_1\supset R_2\). The recursion rests on two
  selection rules, of rather different characters; see Remark~\ref{rem:choices}.

  The two quantities that matter going forward are genuinely different. The \emph{achieved
  direction} recorded for the density argument below is the possibly-signed pair
  \((P_k,Q_k)=(n_1s,n_2t)\). But \(R_k\) itself must serve as the input \(R(n_1',n_2',x_0',y_0')\)
  to the \emph{next} stage's application of Lemma~\ref{lem:steering}, and that lemma's hypotheses
  require positive \(n_1',n_2'\), just as Definition~\ref{def:Rcoset} does; so \(R_k\) must instead
  be written using the \emph{positive} pair
  \[
    n_1'=n_1|s|, \quad n_2'=n_2|t|, \quad x_0'=x_0+n_1u_0, \quad y_0'=y_0+n_2w_0.
  \]
  If we write the internal coordinates as \(u=u_0+|s|u'\), \(w=w_0+|t|w'\) --- exactly what
  ``\(u\equiv u_0\pmod{|s|}\)'' means --- then a direct check against the definition of
  \(R_{u_0,w_0}\)~\eqref{eq:Rsub-def} shows \(R_{u_0,w_0}=R(n_1|s|,n_2|t|,x_0',y_0')\) on the nose.

\begin{figure}[htbp]
  \centering
  \begin{tikzpicture}[scale=0.42]
    \begin{scope}
      \clip (-1.6,-2.6) rectangle (13.6,20.6);
      \foreach \x in {-1,...,13} \foreach \y in {-2,...,20}
        \fill[gray!25] (\x,\y) circle (0.10);
      \draw[gray!55] (-1.6,0) -- (13.6,0);
      \draw[gray!55] (0,-2.6) -- (0,20.6);
      \foreach \x in {0,2,...,12} \foreach \y in {1,4,...,19}
        \fill[gray!75] (\x,\y) circle (0.20);
      \foreach \x in {0,4,8,12} \foreach \y in {1,10,19}
        \fill[blue!65!black] (\x,\y) circle (0.31);
      \draw[red!75!black, line width=0.9pt] (-0.30,-0.30) -- (0.30,0.30);
      \draw[red!75!black, line width=0.9pt] (-0.30,0.30) -- (0.30,-0.30);
    \end{scope}
    \node[red!75!black, font=\scriptsize, right] at (0.5,-1.4) {$z_1=(0,0)$};
    \fill[gray!45]       (14.1,16.0) circle (0.10);
    \node[font=\scriptsize, right] at (14.4,16.0) {$R_0=\Z^2$};
    \fill[gray!75]       (14.1,13.0) circle (0.20);
    \node[font=\scriptsize, right] at (14.4,13.0) {$R_1$};
    \fill[blue!65!black] (14.1,10.0) circle (0.31);
    \node[font=\scriptsize, right] at (14.4,10.0) {$R_2$};
  \end{tikzpicture}
  \caption{The first two stages of the nested chain \(\Z^2=R_0\supset R_1\supset R_2\) built in the
    proof of Theorem~\ref{thm:main}. Small grey dots are \(R_0=\Z^2\); medium dots
    \(R_1=R(2,3,0,1)=\{x\equiv0\bmod2,\ y\equiv1\bmod3\}\); large dots
    \(R_2=R(4,9,0,1)=\{x\equiv0\bmod4,\ y\equiv1\bmod9\}\). Each stage claims everything it does
    \emph{not} reserve, so the region shrinks while the claimed set grows to fill \(\Z^2\).
    The run shown takes \((s,t)=(2,3)\) at both stages and the enumeration point \(z_1=(0,0)\)
    (crossed). Since \(z_1\in R_0\), stage~1 may not reserve the sub-coset containing it: the
    lexicographically least \emph{eligible} residue pair is \((0,1)\) rather than \((0,0)\), which
    is why \(R_1\) consists of \(y\equiv1\) rather than \(y\equiv0\) modulo \(3\). That reservation
    rule is exactly the one the Lean formalization computes; the pairs \((s,t)\) are illustrative
    only --- see Remark~\ref{rem:choices}.}
  \label{fig:nesting}
\end{figure}
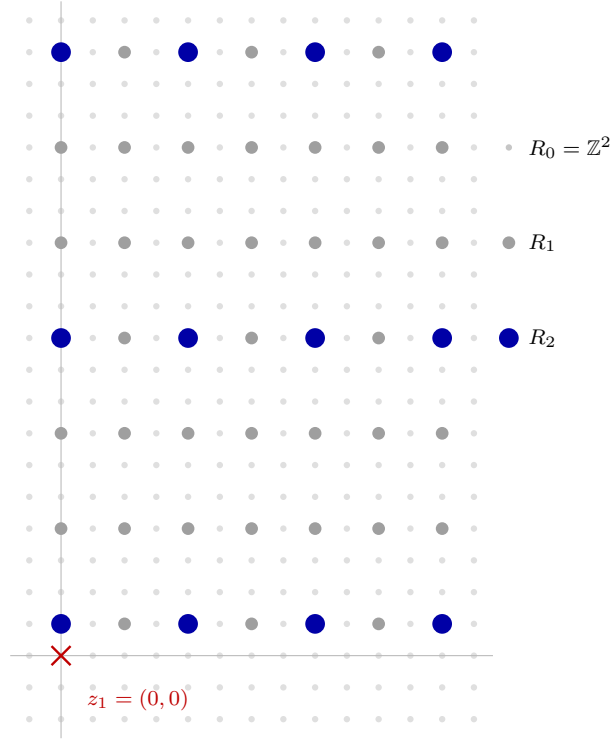

  Write
  \(R_k=R(n_1^{(k)},n_2^{(k)},x_0^{(k)},y_0^{(k)})\) for this representation, i.e.\
  \(n_1^{(k)}=n_1|s|\), \(n_2^{(k)}=n_2|t|\) (with \(n_1^{(0)}=n_2^{(0)}=1\) from the base case), the
  notation used below. This keeps
  \[
    \gcd(n_1^{(k)},n_2^{(k)}) = \gcd(n_1|s|,n_2|t|) = \gcd(P_k,Q_k) = 1
  \]
  (Lemma~\ref{lem:splitting-coset}), so the invariant needed for the next iteration is preserved.

  \emph{Coverage invariant.} Stage \(i\) claims exactly \(R_{i-1}\setminus R_i\), and \(R_i\subseteq
  R_{i-1}\) for every \(i\ge1\). Telescoping from \(R_0=\Z^2\) gives, for every \(k\),
  \begin{equation}
    \label{eq:main-coverage-invariant}
    \Z^2\setminus R_{k-1} \;=\; \bigcup_{i<k}(R_{i-1}\setminus R_i) \;=\; \bigcup_{i<k}
    (\text{stage-}i\text{ claims}),
  \end{equation}
  i.e.\ a point lies outside the current region \(R_{k-1}\) exactly when some earlier stage already
  claimed it. This is what justifies the reservation rule above only ever needing to avoid \(z_k\)
  when \(z_k\in R_{k-1}\): whenever \(z_k\notin R_{k-1}\), \eqref{eq:main-coverage-invariant} shows
  \(z_k\) was already claimed by an earlier stage, so no action is needed for it at stage \(k\).

  \emph{Coverage.} By~\eqref{eq:main-coverage-invariant}, every \(z_k\) is eventually claimed: if
  \(z_k\in R_{k-1}\), the reservation rule ensures \(z_k\) lies in one of the claimed sub-cosets at
  stage \(k\) (not the reserved \(R_k\)); if \(z_k\notin R_{k-1}\), it was already claimed at some
  earlier stage. So
  \begin{equation}
    \label{eq:main-coverage}
    \bigcup_k(\text{stage-}k\text{ claims}) \;=\; \Z^2.
  \end{equation}
  Since the stage-\(k\) claims are exactly \(R_{k-1}\setminus R_k\), and \(R_0=\Z^2\), this says
  precisely that
  \[\bigcap_{k\ge0} R_k=\varnothing,\]
  that is: every point leaves the chain at some \emph{finite} stage. This is what makes the
  recursion self-contained at order type \(\omega\), with no step beyond those indexed by
  \(k\in\N\).

  Nesting alone would not give it. The rule ``always reserve the sub-coset containing the origin''
  produces a chain that is strictly decreasing, with the same densities
  \(1/(n_1^{(k)}n_2^{(k)})\to0\), and yet retains \((0,0)\) in every \(R_k\) --- a point that no line
  of \(\mathcal F\) would ever cover. What excludes this is the reservation rule's dependence on the
  enumeration \((z_k)\), and that works only because every point of \(\Z^2\) carries a finite index.

  \emph{Disjointness.} Every individual line placed into \(\mathcal F\) at stage \(k\) has all of its
  lattice points inside the sub-coset \(R_{u_0,w_0}\) it came from (a subset of \(R_{k-1}\setminus
  R_k\), since \(R_{u_0,w_0}\ne R_k\)): this is exactly~\eqref{eq:splitting-coset-conclusion} of
  Lemma~\ref{lem:splitting-coset} (each line on the left is a subset of the union), the step the
  earlier, flawed version of this construction got wrong (see the remark following
  Lemma~\ref{lem:splitting-coset}). So stage \(k\)'s claimed points lie in
  \[R_{k-1}\setminus  R_k\subset R_{k-1}.\]
  For \(j>k\), \(R_{j-1}\subseteq R_k\) (as \(R_i\subseteq R_{i-1}\) for every \(i\)),
  so stage \(j\)'s points, lying in \(R_{j-1}\), are disjoint from stage \(k\)'s. Hence distinct lines
  from different stages never share a lattice point, and distinct lines of the same direction from
  the same stage are disjoint level sets of the same \(\varphi_{(P_k,Q_k)}\), so never share one
  either.
  Different stages in fact always realize different directions: \(|P_k|=n_1^{(k)}\) is
  strictly increasing in \(k\), since
  \[n_1^{(k)}=n_1^{(k-1)}|s_k|\ge2n_1^{(k-1)}.\]
  The argument above does not need this.

  \emph{Density.} Fix \(\theta\in\RP\) and \(\varepsilon>0\). By Remark~\ref{rem:tail-density}, every
  tail \(\{\theta_k\}_{k\ge N}\) is dense in \(\RP\); so for any \(N\) there is some \(k\ge N\) with
  \(\theta_k\) within \(\varepsilon/2\) of \(\theta\). Taking
  \(N>2/\varepsilon\), the corresponding achieved direction \((P_k,Q_k)\) --- within \(1/k<\varepsilon/2\)
  of \(\theta_k\) --- then lies within \(\varepsilon\) of \(\theta\). Each \((P_k,Q_k)\) is genuinely
  realized in \(\mathcal F\): at least \(|s||t|-1\ge3\) of the \(\ge4\) sub-cosets at stage \(k\) are
  claimed, all nonempty, each contributing lines of that direction. So every \(\theta\in\RP\) is a
  limit of directions realized in \(\mathcal F\): the realized direction set is dense.

\end{proof}

\begin{remark}[The two selection rules]
  \label{rem:choices}
  The recursion in the proof above is not a definition until two things are pinned down: which
  admissible pair \((s,t)\) to take among those Lemma~\ref{lem:steering} merely asserts to exist,
  and which sub-coset to reserve. Neither affects the conclusion --- the argument uses only the
  \emph{properties} of the chosen objects, never their identity --- but a recursion is a
  well-defined function only once both are fixed. Informally one writes ``choose \(s,t\)'' and moves
  on; a formalization cannot, since the recursion has to be an actual term.

  The two are settled in genuinely different ways, and the contrast is worth noticing. The
  reservation rule is \emph{computable}: the reserved pair is the lexicographically least element of
  the finite set of residue pairs \((u_0,w_0)\in[0,|s|)\times[0,|t|)\) whose sub-coset avoids
  \(z_k\) --- a condition that is vacuous when \(z_k\notin R_{k-1}\), which is exactly the
  ``otherwise'' clause of the rule stated in the proof. That set is nonempty because there are at
  least \(|s||t|\ge4\) sub-cosets and, by Lemma~\ref{lem:splitting-partition}, \(z_k\) lies in
  exactly one of them, so at least three remain. The steering pair, by contrast, is obtained from
  the existence statement by classical choice: Lemma~\ref{lem:steering} singles out no canonical
  \((s,t)\), and the proof needs none.

  Why only one of the two is computable is a question of order type rather than of principle. The
  reservation is a minimum over the \emph{finite} set \([0,|s|)\times[0,|t|)\), so
  ``lexicographically least'' commits us to nothing: every total order on a finite set has finite
  order type, and a graded order would serve equally well --- it would generally reserve a different
  sub-coset, but the proof never uses which one. The steering pair, by contrast, is drawn from an
  \emph{infinite} subset of \(\Z^2\), and there the choice of order does matter. Ordered
  lexicographically, \(\Z^2\) has order type \(\omega^2\); a nonempty subset still has a least
  element, so the selection is well defined --- but locating it may require deciding whether an
  entire infinite column meets the set, which no terminating search can do. Under a \emph{graded}
  order --- by \(|s|+|t|\), say, with lexicographic tie-breaking --- the order type is \(\omega\),
  every pair has only finitely many predecessors, and the least admissible pair is found by running
  through candidates in that order until one passes the coprimality tests and the strict inequality
  of Lemma~\ref{lem:steering}. That search halts precisely because the order type is \(\omega\).

  So classical choice here is a convenience, not a necessity: taking the graded-least admissible
  pair makes the steering rule computable too, and with it the recursion as a whole. We keep the
  existential form because that is what the argument uses, and no canonical pair is needed for the
  theorem.\footnote{For term orders on \(\N^n\), which are required in addition to respect addition
  --- a constraint of no use here --- the possible order types are exactly
  \(\omega^1,\dots,\omega^n\). This is a consequence of the classification of such orders: by
  Robbiano's theorem~\cite{Robbiano1985} every one of them is given by a real matrix acting on
  exponent vectors and compared lexicographically, and Robbiano settles as well \emph{when two
  matrices determine the same order}, so the representation is genuinely understood rather than
  merely available. Eisenbud puts the classification in the equivalent form that every monomial
  order is a lexicographic product of at most \(n\) weight orders~\cite[Exercise~15.13]{Eisenbud1995};
  the number of factors is then the exponent of the order type. The moral survives in the present
  looser setting: what a terminating search needs is not compatibility with the group structure, but
  order type \(\omega\).}

  So the construction is a definite object once an enumeration \((z_k)\), a dense sequence
  \((\theta_k)\) and a selection of steering pairs are fixed; only the last of these is genuinely
  arbitrary. Figure~\ref{fig:nesting} shows one such run.
\end{remark}

\section{Pictures of the construction}

Figure~\ref{fig:nesting} above shows the nested regions \(R_k\) that drive the recursion;
this section instead looks at what the construction \emph{achieves}.
Figure~\ref{fig:directions} plots the first \(10\), \(20\), and \(100\) directions realized by an actual
run of the construction (code in \texttt{verification-code/}, targeting the equidistributed sequence
\(\theta_k = k\pi(\sqrt5-1)/2 \bmod \pi\), an irrational rotation by the golden ratio conjugate),
each direction shown as an antipodal pair of points
on the unit circle (a line's direction is unsigned) and colored by its position in the sequence.

\begin{figure}[h]
  \centering
  \includegraphics[width=\textwidth]{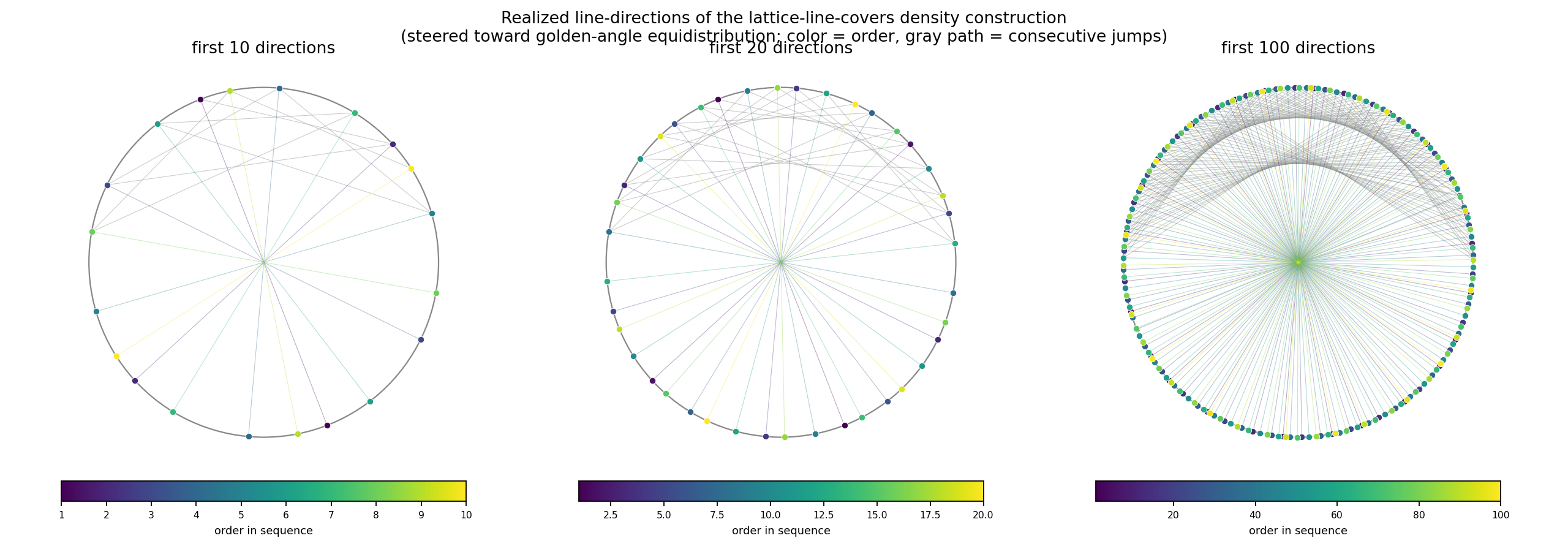}
  \caption{Realized directions after \(10\), \(20\), and \(100\) steps of the construction, colored by
  order (dark = early, light = late), with a thin path tracing consecutive jumps. The steering
  argument scatters the directions rather than sweeps them around the circle, since each target
  angle is chosen from an equidistributed sequence rather than visited in angular order.}
  \label{fig:directions}
\end{figure}

\begin{remark}[Sequence order against angular order]
  \label{rem:order-transverse}
  The points in Figure~\ref{fig:directions} are indexed by the stage \(k\) that produced them ---
  that index is what the color ramp shows --- and not by where they sit on \(\RP\). The two orders
  are necessarily different, and not merely incidentally so: an enumeration of a dense subset of
  \([0,\pi)\) cannot be monotone. If \(\theta_1<\theta_2<\cdots\), then no term at all lies in the
  nonempty open interval \((\theta_1,\theta_2)\), since every later term is at least \(\theta_2\);
  the set of terms therefore misses an open set and is not dense. A decreasing enumeration fails for
  the same reason. Density thus forces the enumeration to jump about, and an equidistributed target
  sequence makes it jump about as evenly as it can. The scatter in the figure is the construction's
  order type \(\omega\) lying transverse to the geometry of \(\RP\), not an artifact of the
  particular run.
\end{remark}

As a purely exploratory aside, Figure~\ref{fig:discmeasure} redraws \(1000\) realized directions
(extending the run behind Figure~\ref{fig:directions}) as point and line masses on the closed unit
disc: each direction contributes a dot at its two antipodal boundary points plus its diameter
chord, weighted by \(w_k=(k+5)^{-0.6}e^{-k/250}\). Unlike a weight such as \(1/k\) or \(1/(k+5)\),
this one is \emph{summable}, so the resulting measure genuinely converges as \(k\to\infty\) to a
fixed, permanently non-uniform limit --- no renormalization needed, and none of the visible
structure is a finite-\(k\) artifact fading toward the uniform measure on \(\RP\). The exponent and
cutoff were chosen so that no small handful of directions dominates the picture: the first \(5\)
directions carry only \(9.3\%\) of the total mass, the first \(20\) carry \(26\%\), the first
\(100\) carry \(62\%\); at \(k=1000\), \(99.6\%\) of the full \(k\to\infty\) total is already
accounted for, so the figure is an honest rendering of the actual limit rather than an arbitrary
truncation of it. Nothing in this manuscript rests on it --- see the Acknowledgements for how it
came about.

\begin{figure}[htbp]
  \centering
  \includegraphics[width=\textwidth]{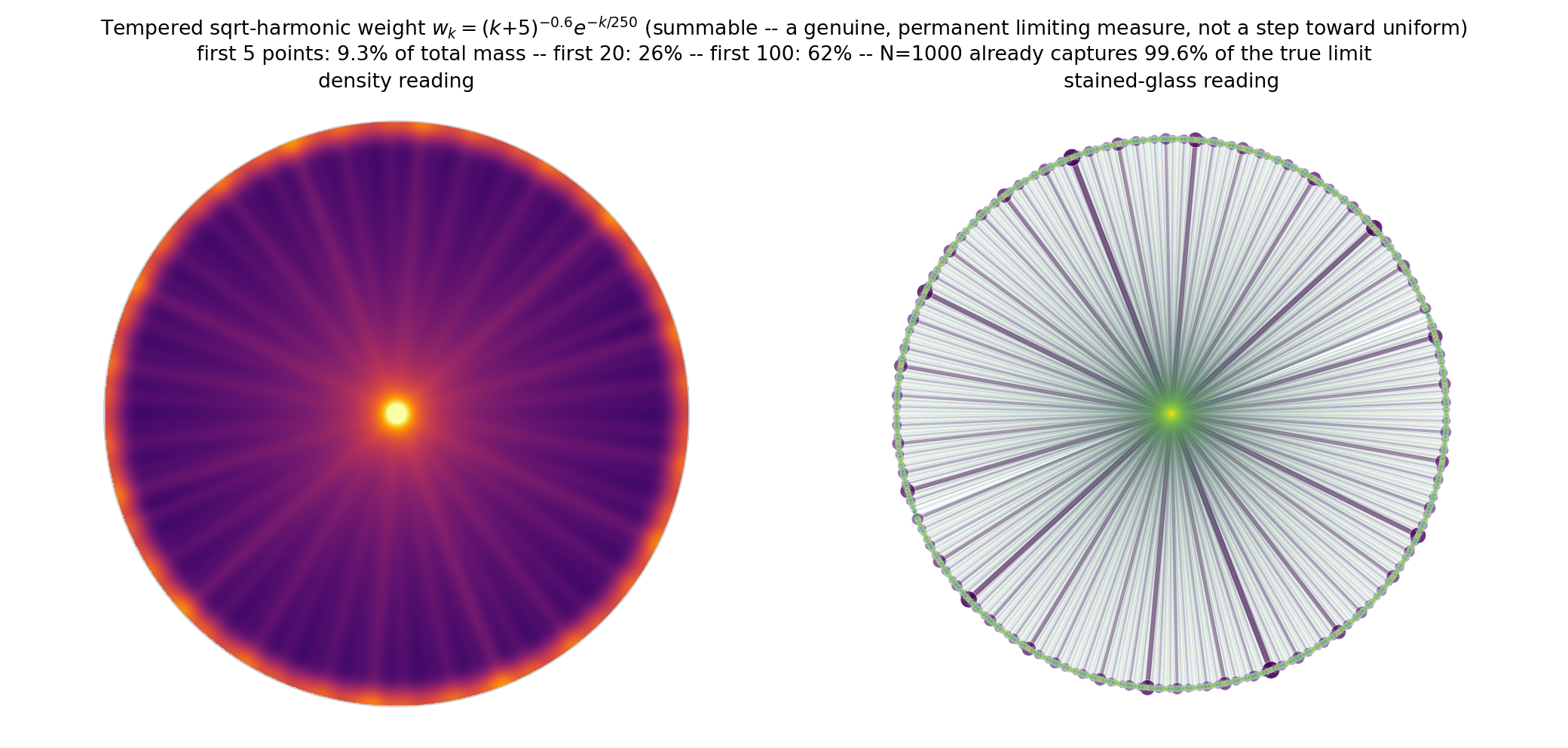}
  \caption{\(1000\) realized directions redrawn as point and line masses on the closed unit disc,
  weighted by the summable \(w_k=(k+5)^{-0.6}e^{-k/250}\) so the picture shows a genuine limiting
  measure rather than a step toward the uniform measure on \(\RP\) (see main text). Left: a
  smoothed density reading. Right: an order-colored overlay of the same weighted points and chords
  (dark purple = early, yellow-green = late). Purely exploratory --- not used anywhere else in this
  manuscript.}
  \label{fig:discmeasure}
\end{figure}

Two further exploratory pictures, again purely illustrative, follow the proof's own mechanism
rather than the directions it realizes. The \emph{coverage invariant} established above
(\eqref{eq:main-coverage-invariant}--\eqref{eq:main-coverage}) says that a point \((u,v)\in\Z^2\)
is claimed by exactly one stage \(k\), namely the unique \(k\) with \((u,v)\in R_{k-1}\setminus
R_k\); write \(\mathrm{stage}(u,v)=k\) for that index. The enumeration \((z_k)\) used throughout
the proof to decide reservations can be any enumeration of \(\Z^2\) in which every point has finite
index; the run behind Figure~\ref{fig:coveragetime} takes the concrete choice of ordering points by
\(L^\infty\)-distance from the origin, breaking ties lexicographically.

\begin{figure}[htbp]
  \centering
  \includegraphics[width=\textwidth]{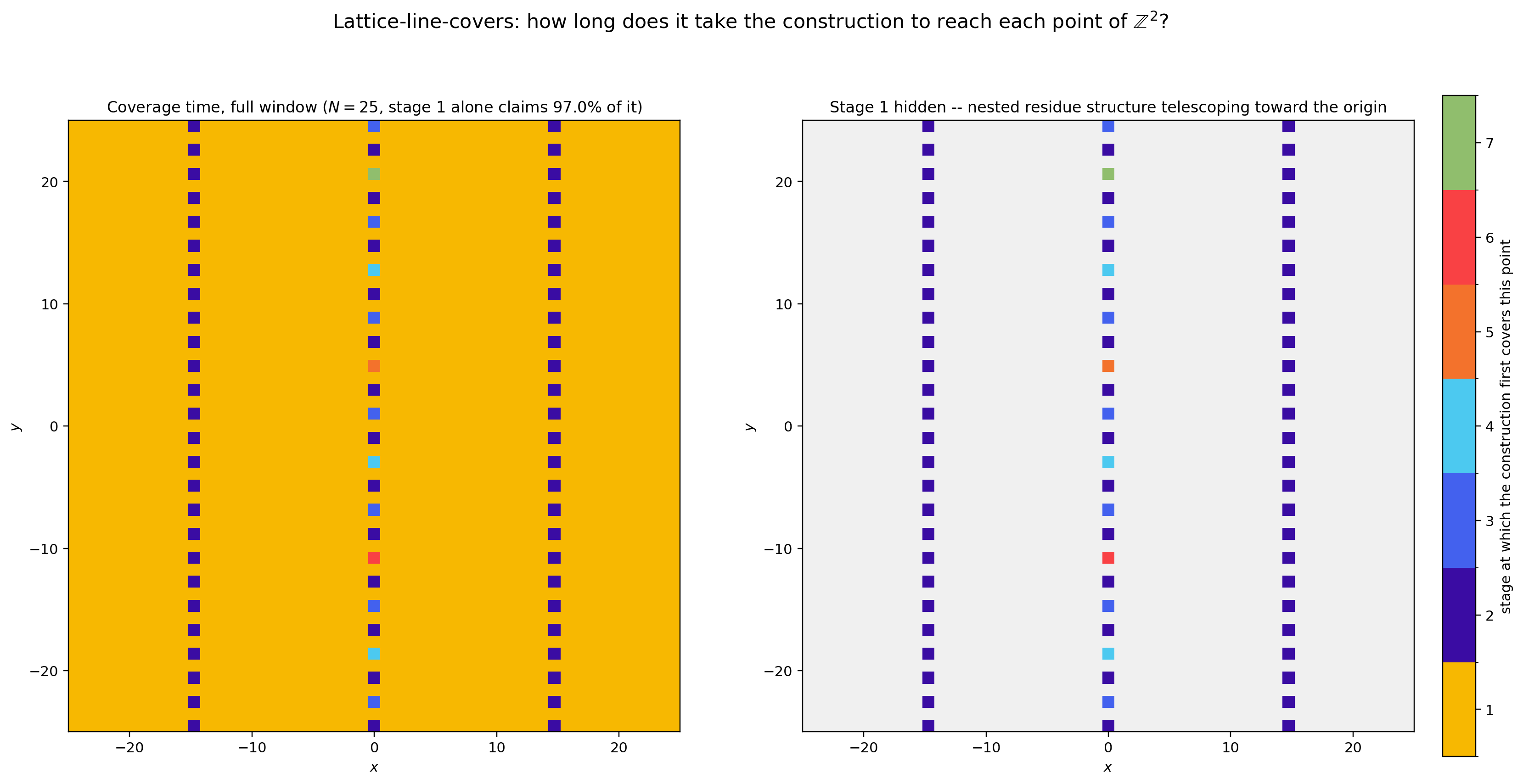}
  \caption{\(\mathrm{stage}(u,v)\) over a \(51\times51\) window (\(N=25\)) of the run behind
  Figure~\ref{fig:directions}. Left: the full window --- stage \(1\) alone (gold) already claims
  \(97.0\%\) of it, since the first direction excludes only a single residue class. Right: stage
  \(1\) hidden, revealing that of the three residual columns visible on the left, only the one
  nearest the origin keeps recursing (through stages \(3\)--\(7\), telescoping inward as \(|y|\)
  grows), while the other two are fully resolved by stage \(2\) alone and never revisited --- a
  direct consequence of \(z_k\) being enumerated by distance from the origin, so whichever residual
  region sits closest is exactly the one reopened next.}
  \label{fig:coveragetime}
\end{figure}

Stage \(k\) claims \(R_{k-1}\setminus R_k\) all at once, via every line \(\ell_{(P_k,Q_k),c}\) of
that stage's direction \((P_k,Q_k)\) needed to cover it (Lemma~\ref{lem:splitting-partition}) --- so
\(\mathrm{stage}\) conflates many distinct lines into one value. A finer invariant tracks the
individual line rather than the stage: enumerate the set \(S=\{\ell_{(P_k,Q_k),c}\}\) of lines
actually used by the construction as \(S=\bigcup_j\{s_j\}\), indexing each line \(s_j=(P,Q,c)\) the
first time the same \(L^\infty\) enumeration \((z_k)\) lands on one of its points, and set
\[
  \mathrm{hit}(u,v) \;=\; \min\{\,j : (u,v)\in\ell_{s_j}\,\}.
\]
Unlike \(\mathrm{stage}\), this has no direct role in the proof --- the construction only ever
needs the coarser, stage-level disjointness --- but it exposes structure the stage-level picture
cannot: at \(N=25\), \(887\) distinct lines are used, of which \(423\) meet the window in exactly
\(3\) points and \(270\) in exactly \(4\), so \(\mathrm{hit}\) is a genuinely fine-grained ordinal,
plotted in Figure~\ref{fig:linehittingtime} on a continuous scale rather than the discrete palette
of Figure~\ref{fig:coveragetime}.

\begin{figure}[htbp]
  \centering
  \makebox[\textwidth][c]{\includegraphics[width=1.32\textwidth]{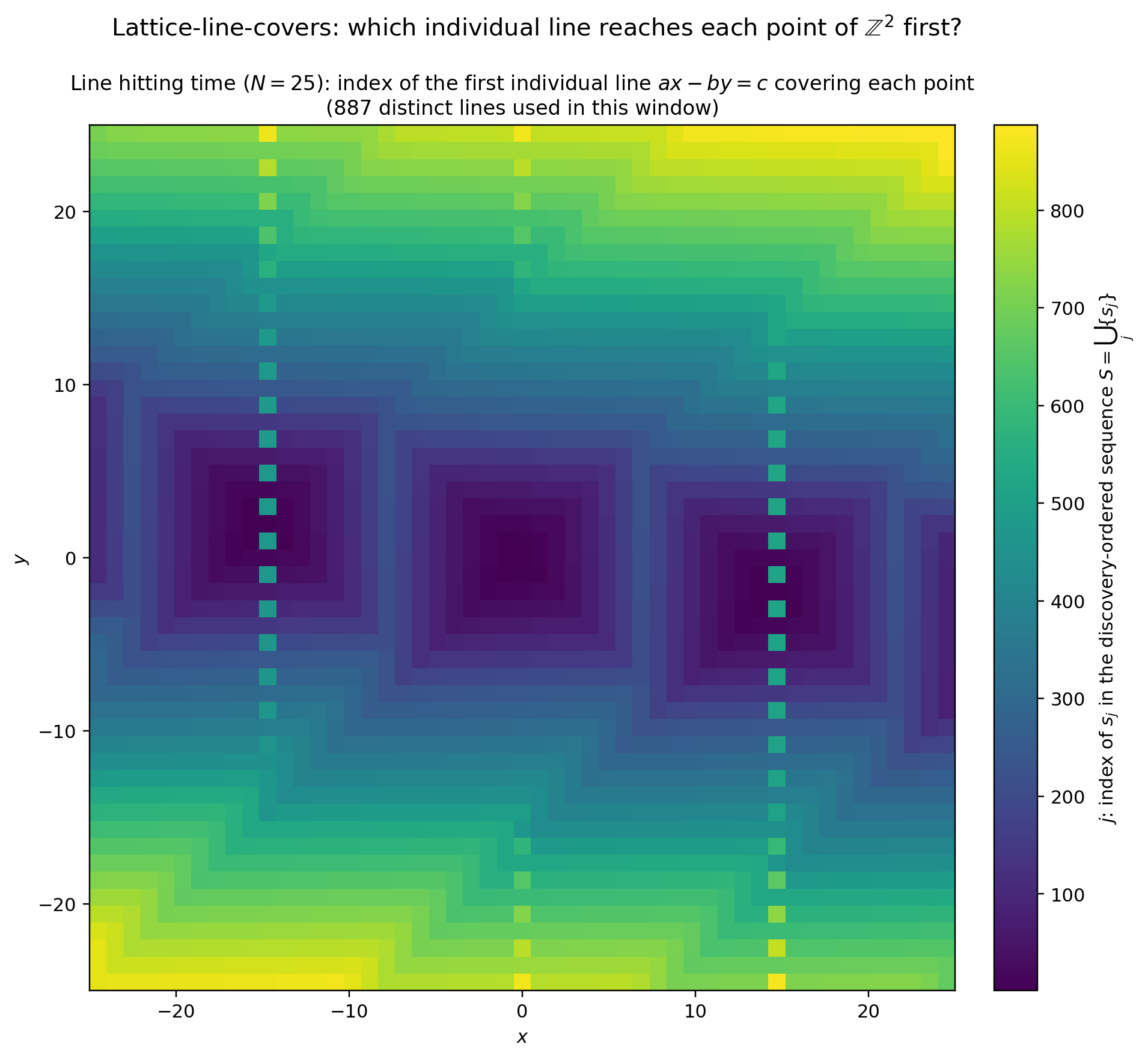}}
  \captionsetup{width=0.94\textwidth}
  \caption{\(\mathrm{hit}(u,v)\) over the same \(51\times51\) window. The concentric \emph{square}
  (not circular) rings are the \(L^\infty\) enumeration made visible. The three low-\(\mathrm{hit}\)
  blobs aligned at \(x=-15,0,15\) are a single mechanism repeated: stage \(1\)'s shallow direction
  \((15,-2)\) places three window-points of the \emph{same} line \(15\) apart in \(x\) (for
  instance \((0,2)\), \((-15,4)\), \((15,0)\) all lie on \(\ell_{(15,-2),-30}\)); since \(\mathrm{hit}\)
  is assigned at a line's first discovery and one representative always sits near the origin, its
  far echoes inherit the same tiny index for free, landing at nearly the same \(|y|\) because the
  direction is so shallow. The thin seam exactly at \(x=0\) is a different mechanism: those points
  are almost always the sole window-representative of a deep-recursion line (coefficients up to
  \(P\approx2.6\times10^6\) at \(N=25\)), so \(\mathrm{hit}\) there tracks \(|y|\) alone rather
  than \(\max(|x|,|y|)\), a slower growth law that reads as a seam cutting through the square
  rings.}
  \label{fig:linehittingtime}
\end{figure}

\section{Related work}

The reformulation via \(\varphi_d\) and the resulting rigidity phenomenon (Lemma~\ref{lem:rigidity})
form the natural two-dimensional analogue of the classical theory of \emph{covering systems of
congruences}, introduced by Erd\H{o}s~\cite{Erdos1950} to show that a positive proportion of the
integers are not of the form \(2^k+p\) for \(p\) prime. A very close relative of our covering problem
--- covering \(\Z^2\) by finitely many finite-index \emph{sublattices} (necessarily containing the
origin, unlike the general affine cosets considered here) --- is studied by Cremona and Koymans
\cite{CremonaKoymans2026}, who give it the same ``projective/homogeneous covering congruences''
framing; their refinement machinery (splitting one lattice into several of prime-power relative
index) is structurally reminiscent of the splitting used here, one rank up.

A different classical result addresses a related but distinct question about how many directions a
covering needs. Corzatt~\cite{Corzatt1985} conjectured that if a \emph{finite, convex} set of lattice
points is covered by \(n\) lines, with no constraint on where the lines may cross, the lines can
always be chosen to use at most four distinct slopes. Our setting is the opposite regime: an
infinite, unbounded point set, with a crossing-avoidance constraint standing in for a bound on the
number of lines, and correspondingly the answer is qualitatively different: densely many
directions are both necessary to consider and available.

The unimodular pairs forbidden by Lemma~\ref{lem:rigidity} are exactly the Farey-neighbor pairs of
slopes; see e.g.\ Hardy and Wright \cite[Ch.~III]{HardyWright2008} for the classical theory of
Farey sequences and the closely related Stern--Brocot tree. This is precisely the toolkit used in
the digital-geometry literature on digitized straight lines, where continued-fraction and
Farey-tree techniques describe digital approximations to a line of given (rational or irrational)
slope; see Kiselman~\cite{Kiselman2022}, and, for the continued-fraction approach specifically,
Uscka-Wehlou's dissertation~\cite{UsckaWehlou2009}.

\begin{openproblem}
  Is it possible to generalize Theorem~\ref{thm:main} to \(\Z^n\) for \(n\ge3\)? In this formulation,
  one would study configurations of ``lattice hyperplanes'' or ``lattice lines'', such that no two
  different directions meet at a lattice point, and determine which sets of directions can occur.
\end{openproblem}

\section*{AI disclosure}

The results in this manuscript --- the reformulation, the rigidity and coset lemmas, the
recursive construction, and both rounds of error-correction (one caught by numerical simulation,
one by an independent review) --- were found by Claude (Anthropic), working with the author in an
interactive session. This manuscript was likewise produced by the author and Claude jointly, with
the author mainly directing Claude toward clearer exposition for human readers: restructuring the
proofs into short, explicitly labeled lemmas, and breaking up long equations for readability. The
author then checked the final manuscript to the best of his ability, made final edits, and assumes
full responsibility for any errors, mistakes, or gaps that remain.

\section*{Description of the work process}

Both the splitting mechanism (Lemmas~\ref{lem:splitting-coset}--\ref{lem:splitting-partition}) and
the sign-generality of Lemma~\ref{lem:steering} were checked by direct brute-force simulation
(exhaustive crossing and coverage checks on finite windows up to \(81\times81\), including
adversarial diagonal and sign-alternating schedules) before being trusted; a splitting step
without the primitivity safeguard built into Lemma~\ref{lem:splitting-coset} was found this way to
be genuinely flawed rather than merely under-justified. The full argument was also read
independently by a second reader (another Claude agent) working from the statement of the lemmas
alone, who found a real gap: restricting \(s,t\) to be positive only reaches half of \(\RP\), and
existence of \emph{some} admissible move at each stage does not imply the ability to steer toward
an \emph{arbitrary} target --- both addressed in the versions of
Lemmas~\ref{lem:splitting-coset}--\ref{lem:splitting-partition} and~\ref{lem:steering} given here.
A machine-checked (Lean/Mathlib) formalization of this argument has also been completed, as a
separate, complementary effort, available at~\cite{gitrepo}; brute-force simulation and
independent reading catch different classes of error than a machine-checked proof does, and vice
versa.

The additional references collected in the repository's extended-reference list \cite{gitrepo}
were located and verified by Claude via web search. A list of ``certificates'' showing how each
reference is used in this manuscript (this ``certificate of non-hallucination''), together with a
more detailed explanation of the whole process of producing this manuscript, from inception to
exploration to proof discovery to proof formalization and verification, are described in the
gitlab repository \cite{gitrepo}; the formalization and verification code is also archived, with a
citable DOI, on Zenodo \cite{zenodo}.

\section*{Acknowledgement}

This problem occurred to the author during an idle conversation with colleagues
F.B.\ and T.E.\ (in the department coffee room) about covering the lattice \(\Z^2\) with lines.
It was considered as a useful test for LLM-assisted proof discovery and proof
formalization, since the author conjectured that the proof of the result,
if it was indeed true, should only use elementary methods, which are already in
lean (Mathlib).

Figure~\ref{fig:discmeasure} was inspired by a remark --- certainly misunderstood by the author ---
about a relation to Besicovitch--Orponen theory, made by his colleague B.E.; the keyword
``measure'' stuck.

\section*{Funding}

This research received no external funding. It was carried out as part of the author's regular
duties at Link\"oping University, within the fraction of that employment allocated to research, the
remainder being teaching and other duties.

\section*{Conflicts of interest}

The author declares no financial or non-financial conflicts of interest. The nature and extent of
LLM assistance in producing this manuscript is disclosed in full in the AI disclosure and
Description of the work process sections above.

\section*{Notation}

Notation used across more than one later result, in order of first appearance; symbols local to a
single proof are omitted.

\begin{center}
\begin{tabular}{@{}p{0.24\textwidth}p{0.48\textwidth}p{0.19\textwidth}@{}}
\hline
\textbf{Symbol} & \textbf{Meaning} & \textbf{Defined in} \\
\hline
\(\RP\) & space of line directions, \(\theta\in[0,\pi)\) & \S1, after Def.~\ref{def:direction} \\
\(d=(p,q)\) & primitive direction, \(\gcd(p,q)=1\) & Def.~\ref{def:direction} \\
\(\varphi_d(x,y)=qx-py\) & direction homomorphism \(\Z^2\to\Z\) & Def.~\ref{def:direction},
  \eqref{eq:phi-def} \\
\(\ell_{d,c}\) & lattice line, \(\varphi_d^{-1}(c)\) & Def.~\ref{def:lline}, \eqref{eq:lline-def} \\
\(\mathcal F\) & covering family & Def.~\ref{def:family} \\
\(\Delta=p_1q_2-p_2q_1\) & determinant of two directions & Lemma~\ref{lem:rigidity},
  \eqref{eq:Delta-def} \\
\(R(n_1,n_2,x_0,y_0)\) & rank-2 coset in \(\Z^2\) & Def.~\ref{def:Rcoset}, \eqref{eq:Rcoset-def} \\
\(R_{u_0,w_0}\) & sub-coset of \(R\) & \eqref{eq:Rsub-def} \\
\((P,Q)\), \((P_k,Q_k)\) & achieved direction of a splitting stage & Lemma~\ref{lem:splitting-coset},
  \eqref{eq:PQ-def} \\
\(\omega(n_1)\) & number of distinct prime divisors of \(n_1\) & \S3, proof of
  Lemma~\ref{lem:steering} \\
\(\mathrm{stage}(u,v)\) & the stage \(k\) that first claims \((u,v)\) & \S4 \\
\(\mathrm{hit}(u,v)\) & index of the first individual line covering \((u,v)\) & \S4 \\
\hline
\end{tabular}
\end{center}


\begin{thebibliography}{9}

\bibitem{CremonaKoymans2026}
J.~E. Cremona and P. Koymans, \emph{Lattice Coverings and Homogeneous Covering Congruences},
arXiv:2601.03212v2 (2026).

\bibitem{Corzatt1985}
C.~E. Corzatt, \emph{Covering convex sets of lattice points with straight lines}, Congressus
Numerantium \textbf{50} (1985), 129--135.

\bibitem{Erdos1950}
P. Erd\H{o}s, \emph{On integers of the form \(2^k+p\) and some related problems}, Summa Brasil.
Math. \textbf{2} (1950), 113--123.

\bibitem{HardyWright2008}
G.~H. Hardy and E.~M. Wright, \emph{An Introduction to the Theory of Numbers}, 6th ed., revised
by D.~R. Heath-Brown and J.~H. Silverman, Oxford University Press, Oxford, 2008.

\bibitem{Kiselman2022}
C.~O. Kiselman, \emph{Elements of Digital Geometry, Mathematical Morphology, and Discrete
Optimization}, World Scientific, Singapore, 2022.

\bibitem{UsckaWehlou2009}
H. Uscka-Wehlou, \emph{Digital Lines, Sturmian Words, and Continued Fractions}, Ph.D. thesis,
Uppsala Dissertations in Mathematics \textbf{65}, Uppsala University, 2009.

\bibitem{Eisenbud1995}
D. Eisenbud, \emph{Commutative Algebra with a View Toward Algebraic Geometry}, Graduate Texts in
Mathematics \textbf{150}, Springer-Verlag, New York, 1995.

\bibitem{Robbiano1985}
L. Robbiano, \emph{Term orderings on the polynomial ring}, in: Proceedings of EUROCAL '85, Vol.~2
(Linz, 1985), Lecture Notes in Computer Science \textbf{204}, Springer, Berlin, 1985, pp.~513--517.

\bibitem{gitrepo}
J. Snellman and Claude (Anthropic), \emph{lattice-line-covers} (source repository, including the
Lean/Mathlib formalization and reference certificates), 2026,
\url{https://gitlab.liu.se/jansn19/lattice-line-covers}.

\bibitem{report2026}
J. Snellman, \emph{On the directions occurring in lattice-line coverings of the integer plane},
LiTH-MAT-R, ISSN 0348-2960, No. 2026:02, Link\"oping University Electronic Press, 2026.
\doi{10.3384/LiTH-MAT-R-2026-02}. Open access via DiVA, \texttt{urn:nbn:se:liu:diva-226766}.

\bibitem{zenodo}
J. Snellman, \emph{lattice-line-covers: Lean/Mathlib formalization and verification code}, Zenodo,
Link\"oping University community, 2026. Concept DOI (all versions) \doi{10.5281/zenodo.21916621};
this snapshot \doi{10.5281/zenodo.21916622}.

\end{thebibliography}
\end{document}